\documentclass[11pt]{iopart}

\usepackage{iopams}
\usepackage{framed}
\usepackage{booktabs}

\usepackage{enumitem,wasysym}
\setitemize{label=\small{\ding{117}}}

\usepackage{psfrag}
\usepackage{pifont}
\usepackage[]{units}
 \usepackage[color=green!10]{todonotes}

\usepackage{amsthm}
\newtheorem{theorem}{Theorem}

\theoremstyle{definition}
\newtheorem{definition}[theorem]{Definition}

\newtheorem{example}[theorem]{Example}

\newtheorem{alg}{Algorithm}

\newcommand{\abs}[1]{\vert#1\vert}

\newcommand\norm[1]{\Vert#1\Vert}
\newcommand{\enorm}{\norm{\edot}}
\newcommand{\set}[1]{\{#1\}}
\newcommand{\kl}[1]{(#1)}
\newcommand{\ekl}[1]{[#1]}
\newcommand{\coloneqq}{:=}

\newcommand{\eps}{\epsilon}

\newcommand{\R}{\mathbb R}

\newcommand{\N}{\mathbb N}

\newcommand{\B}{\mathbf B}
\newcommand{\A}{\mathbf A}

\newcommand{\Ho}{\mathbf H}
\newcommand{\Ps}{\mathbf P}
\newcommand{\Qo}{\mathcal W}
\newcommand{\x}{\mathbf x}
\newcommand{\rr}{\mathbf r}

\newcommand{\source}{p_0}
\newcommand{\z}{\mathbf z}
\newcommand{\ee}{\mathbf e}
\newcommand{\p}{\mathbf p}
\newcommand{\q}{\mathbf q}
\newcommand{\y}{\mathbf y}
\newcommand{\edot}{\, \cdot \,}

\newcommand{\uu}{x}
\newcommand{\vv}{y}
\newcommand{\ww}{z}

\usepackage{amsopn}

\newcommand{\T}{\mathbf T}

\makeatletter
\newcommand*\bigcdot{\mathpalette\bigcdot@{.6}}
\newcommand*\bigcdot@[2]{\mathbin{\vcenter{\hbox{\scalebox{#2}{$\m@th#1\bullet$}}}}}
\makeatother

\begin{document}

\title{Compressed sensing and sparsity in photoacoustic tomography}

\author{Markus Haltmeier{}$^1$, Thomas Berer{}$^{2}$, Sunghwan Moon{}$^3$ and  Peter Burgholzer{}$^{2,4}$}

\address{{}$^1$ Department of Mathematics, University of Innsbruck, Technikerstrasse 13, A-6020 Innsbruck, Austria}

\vspace{0.7em}

\address{{}$^2$ Research Center for Non-Destructive Testing (RECENDT), Altenberger Stra{\ss}e 69, 4040 Linz, Austria.}

\vspace{0.7em}

\address{{}$^3$ Department of Mathematical Sciences,
Ulsan National Institute of Science and Technology, Ulsan 44919, Republic of Korea}

\vspace{0.7em}

\address{{}$^4$ Christian Doppler Laboratory for Photoacoustic Imaging and Laser Ultrasonics, Altenberger Stra{\ss}e 69, 4040 Linz, Austria.}

\vspace{1em}

\ead{markus.haltmeier@uibk.ac.at}

\vspace{10pt}

\begin{abstract}
Increasing  the imaging speed  is a  central aim in photoacoustic tomography.
This issue is especially important in the case of sequential scanning approaches as applied for most existing optical detection schemes.
In this work we address this issue using  techniques of compressed sensing.
We demonstrate, that the number of measurements can significantly be  reduced  by allowing general linear measurements instead of point-wise pressure values.
A main requirement in compressed sensing is the sparsity of the unknowns to be recovered. For that purpose we develop the concept of sparsifying temporal transforms for three-dimensional photoacoustic tomography. We establish a two-stage algorithm that recovers the complete pressure signals in a first step and then applies a standard reconstruction algorithm  such as back-projection. This yields a novel reconstruction method with much lower complexity than existing compressed sensing approaches for photoacoustic tomography.
Reconstruction results for simulated and for experimental data verify that the proposed compressed sensing scheme allows to significantly reducing the number of spatial measurements without reducing the spatial resolution.

\end{abstract}


\vspace{2pc}

\section{Introduction}
\label{sec:intro}

Photoacoustic tomography (PAT), also known as optoacoustic tomography,  is a novel non-invasive imaging technology that beneficial combines the high contrast of pure optical imaging with the high spatial resolution of pure ultrasound imaging (see \cite{Bea11,Wan09b,XuWan06}).
The basic principle of PAT is  as follows (compare Figure~\ref{fig:pat}). A semitransparent sample (such as a part of a human patient) is illuminated with short pulses of optical radiation. A fraction  of the optical energy is absorbed inside the sample which causes thermal heating, expansion, and a subsequent acoustic pressure wave  depending on the interior absorbing structure of the sample.
The acoustic pressure is measured outside of the  sample and used  to reconstruct an image of the interior.

\psfrag{L}{\scriptsize illumination}
\psfrag{A}{\scriptsize  sound sources}
\psfrag{P}{\scriptsize  sensor locations}
\begin{figure}[thb!]
\begin{centering}
\includegraphics[width=0.75\columnwidth]{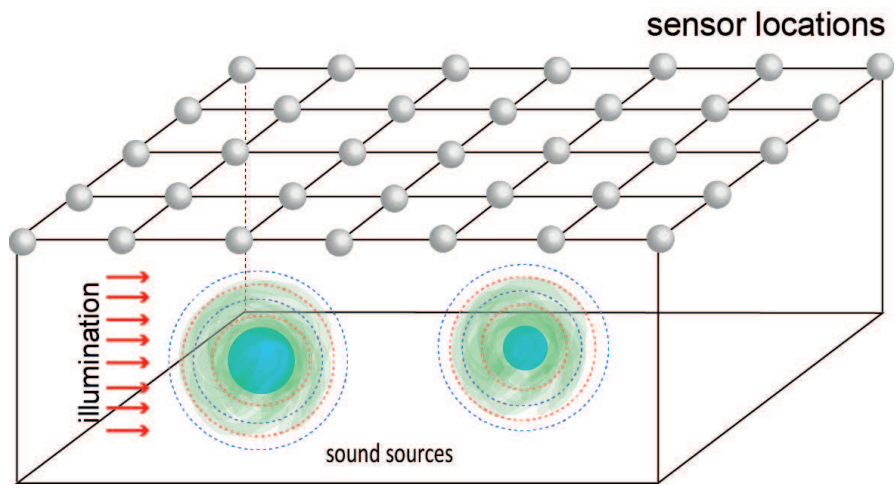}
\caption{\textsc{Basic setup of PAT.} An object  is illuminated with a  short optical pulse that  induces an acoustic pressure wave. The pressure wave is measured on discrete locations on a surface and used to reconstruct an image of the interior absorbing structure. The small spheres indicate the possible detector or sensor locations on a regular grid on the measurement surface. \label{fig:pat} }
\end{centering}
\end{figure}

\subsection{Classical measurement approaches}

The standard approach in PAT is to measure the  acoustic pressure  with small detector elements  distributed on a  surface outside of the sample; see   Figure~\ref{fig:pat}.   The spatial sampling step size  limits the spatial resolution of the  pressure data   and the (lateral) resolution of the final reconstruction.\footnote{Note that  there are several other important factors limiting the resolution of PAT, such as finite detector size, limited  detection bandwidth, a limited acoustic aperture,  or acoustic attenuation.} Consequently, high spatial resolution requires a large number of detector  locations. Ideally, for high frame rate, the pressure data are measured in parallel with a large array made of small  detector elements.  However, the
signal-to-noise ratio and therefore the sensitivity decreases
for smaller detector elements and producing a large  array with high bandwidth is costly and difficult to fabricate.

As an alternative to the usually employed piezoelectric transducers, optical detection schemes have been used to acquire the pressure data \cite{ZhaLauBea08,Berer10,Berer15,Eom15}. In these methods an optical beam is raster scanned along a surface. In case of non-contact photoacoustic imaging schemes the ultrasonic waves impinging on the sample surface change the phase of the reflected light, which is demodulated by interferometric means and a photodetector \cite{Berer10,Berer15,Eom15}. For Fabry-Perot film sensors, acoustically induced changes of the optical thickness of the sensor lead to a change in the reflectivity, which can be measured using a photo diode \cite{ZhaLauBea08}. Equally for both techniques, the ultrasonic data are acquired at the location of the interrogation beam by recording the time-varying output of the photodetector. In order to collect sufficient data the measurement process has to be repeated with changed locations of the interrogation beam. Obviously, such an approach slows down the imaging speed. The imaging speed can be increased by multiplying the number of interrogation beams. For example, for a planar Fabry-Perot sensor a detection scheme using 8 interrogation beams has been  demonstrated in \cite{huynh2016photoacoustic}.

Another, less straight forward, approach to increase the measurement speed is the use of patterned interrogation together with compressed sensing techniques.  Patterned interrogation was experimentally demonstrated using a digital micromirror device (DMD) in \cite{HuyZhaBetEtal14,huynh2016single}. Using digital micromirror devices or spatial light modulators to generate such interrogation patterns together with compressed sensing techniques allows to reduce the number of spatial measurements without significantly increasing the production costs. For such approaches, we develop a compressed sensing scheme based 
sparsifying temporal transforms originally introduced  for PAT with integrating line detectors in~\cite{SanKraBerBurHal15,burgholzer2016}.

\subsection{Compressed sensing}

Compressed sensing (or compressive sampling) is a new sensing paradigm introduced  recently in \cite{CanRomTao06a,CanTao06,Don06}.
It allows to capture high resolution signals using   much less measurements than advised by Shannon's sampling theory.
The basic idea in compressed sensing is replacing point measurements by general linear measurements, where each measurement consists of a linear combination
\begin{equation} \label{eq:cs}
    \y[j]
    = \sum_{i=1}^n \A\ekl{j,i}  \x\ekl{i}
    \quad \textnormal{ for } j =  1, \dots , m  \,.
\end{equation}
Here $\x$  is the desired high resolution signal (or image),
$ \y $ the measurement vector, and  $\A$  the $m \times n$ measurement matrix.
If  $m  \ll n$, then (\ref{eq:cs}) is a severely
under-determinated system of linear equations for the unknown signal. The theory of compressed sensing  predicts that under suitable assumptions the unknown signal can nevertheless be stably recovered from such  data.
The crucial ingredients of compressed sensing are sparsity and randomness.
\begin{enumerate}
\item \textsc{Sparsity:} This refers to the requirement that the unknown signal is sparse, in the sense that it has only a small number of
entries  that are significantly different from zero (possibly after a change of basis).

\item \textsc{Randomness:} This  refers to selecting the entries of the  measurement matrix
in a certain random fashion. This guarantees that the measurement data
are able to sufficiently well separate sparse vectors.
\end{enumerate}

\noindent  In this work we use randomness and sparsity to develop
novel compressed sensing techniques for PAT.

\begin{figure}[thb]
\begin{centering}
\includegraphics[width=0.75\columnwidth]{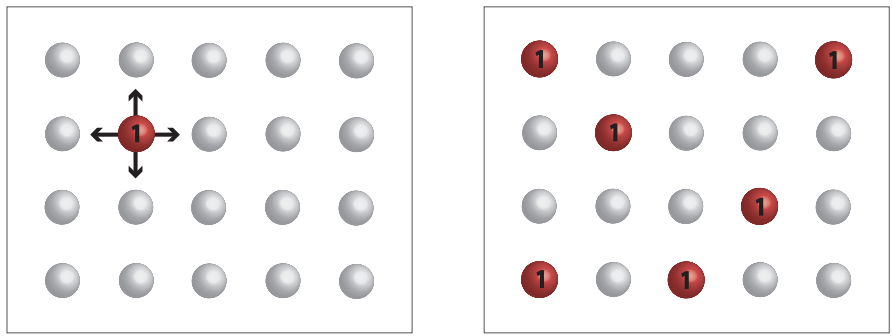}
\caption{\textsc{Standard sampling versus compressed sensing.} Left: Standard sampling records  point-wise data
at individual  detector positions. Right:  Compressed sensing measurements consist of  random combinations of point-wise data  values.\label{fig:sampling} }
\end{centering}
\end{figure}

\subsection{Compressed sensing in PAT}

In  PAT, temporal samples can easily be collected at a high rate compared to spatial sampling, where each sample requires a separate sensor. It is therefore natural to work with semi-discrete data $p(\rr_S[i],  \edot) $, where $\rr_S[i]$ denote  locations on the detection surface. Compressed sensing measurements in PAT  take the form  (\ref{eq:cs})  with $\x\ekl{i}  \coloneqq p(\rr_S[i],  t) $ for fixed time $t$. See Figure~\ref{fig:sampling} for an illustration of classical point-wise sampling versus  compressed sensing measurements. In PAT it is most simple to use binary combinations of pressure values, where $\A\ekl{j,i}$ only  takes two values (states on and off).   Binary measurements can be implemented by optical detection using  patterned interrogation and we restrict ourselves to such a situation.

In the PAT literature, two  types of binary matrices allowing compressed sensing have been  proposed (see  Figure~\ref{fig:matrices}).
In~\cite{HuyZhaBetEtal14,huynh2016single} scrambled Hadamard matrices  have been used   and experimentally realized. In~\cite{SanKraBerBurHal15,burgholzer2016}  expander matrices have been used, where  the measurement matrix is sparse and  has exactly $d$ ones in each column, whose locations are randomly selected. Another possible choice would be a Bernoulli matrix where any entry is selected randomly  from two values with equal probability.  In  all three cases, the random nature of the selected coefficients  yields compressed sensing capability of the measurement matrix (see~\ref{sec:cs} for details).  
As in \cite{SanKraBerBurHal15,burgholzer2016}, in this study we use expander matrices.  For the experimental verification such measurements are implemented virtually by taking full point-measurements in the experiment and then  computing compressed sensing data numerically. This can be seen as proof of principle;  implementing pattern interrogation in our contact-free photoacoustic imaging device is  an important future aspect.

\begin{figure}[thb!]
\begin{centering}
\includegraphics[width=\columnwidth]{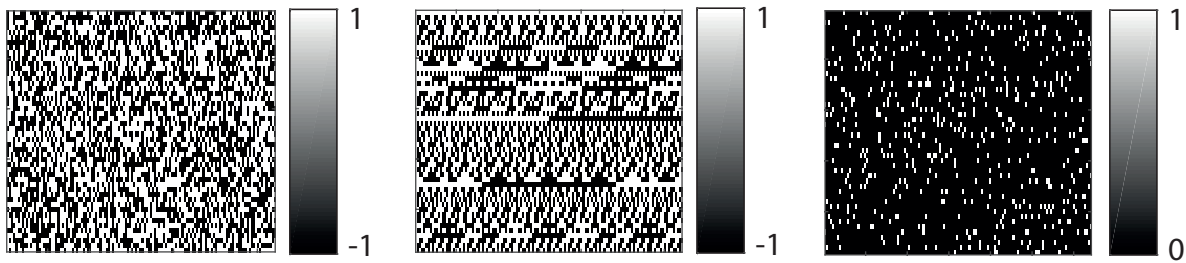}
\caption{\textsc{Binary random matrices allowing compressed sensing.} Left:  Bernoulli matrix is dense and unstructured. Center: Subsampled Hadamard matrix is dense and structured.  Right:   Expander matrix is sparse and unstructured. See~\ref{sec:cs} for more details how to construct these matrices.\label{fig:matrices} }
\end{centering}
\end{figure}

Besides the random nature of the measurement matrix, sparsity of the signal to be recovered is  the 
second  main ingredient  enabling compressed sensing. As in many other applications, sparsity often does not hold in the original domain. Instead sparsity holds in a particular orthonormal basis, such as  a wavelet or curvelet basis \cite{CanDemDonYin06,Mal09}.
However, such a  change  of basis can destroy the compressed sensing capability of the measurement matrix (for example, in the case of expander matrices).
In order to overcome this limitation, in \cite{SanKraBerBurHal15,burgholzer2016} we developed the concept of a  sparsifying  temporal transformation. Such a transform applies in the temporal variable only and results in a
 filtered pressure signal that is sparse.  Because any operation acting in the temporal  domain intertwines with the measurement matrix, one can apply sparse recovery to estimate the sparsified pressure.
The photoacoustic source can be recovered, in a second step,
by applying a standard reconstruction algorithm to the sparsified pressure.

\subsection{Outline of this paper}

In this paper we develop a compressed sensing scheme based on a sparsifying
transform for three-dimensional PAT (see Section~\ref{sec:cspat}).
This complements our work~\cite{SanKraBerBurHal15,burgholzer2016}, where we introduced the concept of sparsifying transforms for PAT with integrating line detectors. Wave propagation is significantly different in two and in three spatial dimensions. As a result, the sparsifying transform proposed in  this work significantly differs from the one presented in~\cite{SanKraBerBurHal15,burgholzer2016}.
In~\ref{sec:cs}  we provide an introduction to compressed sensing serving as  guideline for designing compressed sensing matrices and highlighting  the role of sparsity. In Section~\ref{sec:num} we present numerical results on simulated as well as on experimental data from a non-contact photoacoustic imaging setup~\cite{Hochreiner13}. These results indicate that the number of spatial measurements  can be reduced by at least  a factor of 4 compared to  the classical point sampling approach. The paper concludes with a discussion presented in Section~\ref{sec:discussion} and a short summary in Section~\ref{sec:conclusion}.

\section{Compressed sensing for PAT in planar geometry}
\label{sec:cspat}

In this section we develop a compressed sensing scheme for PAT,
where the acoustic signals are recorded on a planar  measurement  surface.
The planar geometry is of particular interest since it is the naturally occurring geometry if using optical detection schemes like the Fabry-Perot sensor or non-contact imaging schemes.  We thereby extend  the concept of sparsifying temporal transforms introduced  for two-dimensional wave propagation in   \cite{SanKraBerBurHal15,burgholzer2016}.
We emphasize that  the proposed  sparsifying transform for the three-dimensional  wave equation can be used for any detection geometry. An extension of our approach to general geometry would, however,  complicate the notation.

\subsection{PAT in planar geometry}

Suppose the  photoacoustic source distribution $\source(\rr)$   is located in the upper half space  $\set{ (\uu, \vv, \ww) \in \R^3 \mid  \ww > 0}$.
The induced acoustic pressure $p(\rr, t)$ satisfies the wave equation
\begin{eqnarray} \label{eq:wave}
   \frac{1}{c^2}  \frac{\partial^2 p(\rr, t)}{\partial t^2}  - \Delta_\rr p(\rr, t)  = - \frac{\partial \delta}{\partial t}(t) \, \source (\rr) \,,
\end{eqnarray}
where  $\Delta_{\rr}$ denotes the spatial Laplacian, $\partial/\partial t$ is the derivative with respect to time, $c$ the sound velocity, and $\delta (t)$ the Dirac delta-function. Here $(\partial \delta/ \partial t)  \source $ acts as the sound source at time $t=0$ and it is supposed that $p(\rr, t) =0$ for $t<0$. We further denote  by
\begin{equation*}
(\Qo \source) (\uu_S, \vv_S, t) \coloneqq p  (\uu_S, \vv_S, 0, t)\\
 \,,
\end{equation*}
the pressure data restricted to the  measurement plane.
PAT in planar recording geometry is concerned with reconstructing $\source$ from measurements of  $ \Qo\source$.

For recovering $ \source$  from continuous data explicit and stable inversion formulas, either in the  Fourier domain or in the time domain, are  well known.  A particularly useful  inversion method is the universal backprojection (UBP),
\begin{equation}\label{eq:ubp3d}
           \source(\rr)    =
        \frac{\ww}{\pi}
        \int_{\R^2}
         (t^{-1} \partial_t t^{-1} \Qo\source) \left(\uu_S,\vv_S,
         \abs{\rr-\rr_S} \right)
        \rmd S \,.
\end{equation}
Here $\rr = (\uu, \vv, \ww)$ is a reconstruction point,
$\rr_S = (\uu_S, \vv_S, 0)$ a point on the detector surface, and  $\abs{\rr-\rr_S} $
the distance between $\rr$ and $\rr_S$.
The UBP has been derived
in~\cite{XuWan05} for planar, spherical and cylindrical geometries.
The two-dimensional version of the UBP
 \begin{equation*}
          \source(\rr)
         =
        - \frac{2 \ww}{\pi}
        \int_{\R}
        \int_{\abs{\rr-\rr_S}}^\infty
        \frac{ (\partial_t t^{-1} \Qo \source)(\uu_S, t)}{ \sqrt{t^2-\abs{\rr-\rr_S}^2}} \rmd t
        \rmd S \,,
\end{equation*}
where $\rr = (\uu,  \ww)$ and $\rr_S = (\uu_S,  0)$ has been first obtained in  \cite{BurBauGruHalPal07}.
In the recent years, the  UBP has been generalized to elliptical observation surface  in two  and three  spatial dimensions \cite{Hal13a,Nat12}, and various geometries in arbitrary dimension
(see~\cite{Kun07a,Hal14,HalPer15b}).

\subsection{Standard sampling approach}

In practical application, only a  discrete number of  spatial measurements can be made.
The standard sensing approach in PAT  is to distribute  detector locations uniformly on a part of the observation surface.  Such data can be modeled by
\begin{equation} \label{eq:pat}
    \p \ekl{i, \edot}  := (\Qo \source) (\uu_S[i], \vv_S[i],  \edot) \quad \textnormal{ for } i = 1, \dots, n  \,.
\end{equation}
The  UBP algorithm applied to  semi-discrete data~(\ref{eq:pat}) consists in discretizing the spatial integral in  (\ref{eq:ubp3d})  using a  discrete sum over all detector locations and evaluating it for a discrete number  of reconstruction points. This yields to the following UBP reconstruction algorithm.

\begin{framed}
\begin{alg}[UBP algorithm for PAT]\label{alg:ubp}\mbox{}\\
\underline{Goal:}
Recover the source  $\source$ in~(\ref{eq:wave}) from data~(\ref{eq:pat}).

\begin{enumerate}[label=(S\arabic*), leftmargin=3em]
\item\label{alg:ubp1}
Filtration: For any $i$, $t$ compute\\
$\q[i, t] \gets  \partial_t t^{-1} \partial_t t^{-1} \p[i, t]$.

\item\label{alg:ubp2}
Backprojection: For  any $k$ set\\
$\p_0 [k]\gets    v[k]  / \pi
\sum_{i=1}^{N^3}  \q \bigl[ i, \abs{\rr[k]-\rr_S[i]}] w_{i}$.
\end{enumerate}
\end{alg}
\end{framed}

In Algorithm~\ref{alg:ubp}, the  first step \ref{alg:ubp1} can be interpreted as  temporal filtering operation.   The second step \ref{alg:ubp2} discretizes the spatial integral in   (\ref{eq:ubp3d})  and is called discrete backprojection.  The numbers $w_i$ are weights for  the numerical integration and account for density of the detector elements.

\subsection{Compressed sensing approach}

Instead of using  point-wise samples, the proposed compressed sensing approach uses  linear combinations  of  pressure values
\begin{eqnarray} \label{eq:data-cspat}
	\y[j,\edot]    =
	\sum_{i=1}^n  \A[j,i] \, \p\ekl{i,\edot}
\quad \textnormal{for }  j \in \set{ 1, \dots,  m } \,,
\end{eqnarray}
where $\A $ is  a binary  $m \times n$ random matrix,
and  $\p \ekl{i, t}$ are  point-wise pressure data.
In the case of compressed sensing we have $m \ll n$, which means that the number of  measurements is much smaller than the number of point-samples. As shown in~\ref{sec:cs}, Bernoulli matrices, subsampled Hadamard  matrices as well as expander matrices  are possible compressed sensing matrices.

In order to recover the photoacoustic source from compressed sensing data   (\ref{eq:data-cspat}), one can use the following  two-stage procedure.
In the first step we recover the point-wise pressure values
from the compressed sensing measurements. In the second step, one applies a standard reconstruction procedure  (such as the UBP  Algorithm~\ref{alg:ubp}) to the estimated point-wise pressure to obtain  the photoacoustic source.
The first step can  be implemented  by  setting $\hat \p[\edot, t] \coloneqq \bPsi \hat \x [\edot, t]$, where
$\hat \x [\edot, t]$ minimizes the $\ell^1$-Tikhonov functional
\begin{equation} \label{eq:ell1relax}
\frac{1}{2}  \norm{\y[\edot,t] - \A \bPsi \hat \x    }^2
 +  \lambda  \norm{  \hat \x  }_1  \to \min_{\hat \x }\,.
 \end{equation}
Here  $\bPsi \in \R^{n \times n}$ is  a suitable basis  (such as orthonormal wavelets) that sparsely represents the pressure data and $\lambda$ is a regularization parameter. Note that (\ref{eq:ell1relax})   can be solved separately for  every $t \in [0, T]$ which makes the two-stage approach particularly efficient.   The resulting two-stage reconstruction scheme is summarized  in Algorithm \ref{alg:cs-pat}.

\begin{framed}
\begin{alg}[Two-stage compressed sensing reconstruction scheme]\label{alg:cs-pat}\mbox{}\\
\underline{Goal:}
Recover $\source$ from data~(\ref{eq:data-cspat}).

\begin{enumerate}[label=(S\arabic*), leftmargin=3em]

\item\label{it:cs-pat1}
Recovery of point-measurements:
\begin{itemize}
\item Choose  a sparsifying  basis $\bPsi \in \R^{n \times n}$.

\item
For every $t$, find an approximation   $\hat \p[\edot, t] \coloneqq \bPsi \hat \x [\edot, t]$ by minimizing (\ref{eq:ell1relax}).
\end{itemize}

\item\label{it:cs-pat2}
Recover $\source$ by applying  a PAT standard reconstruction algorithm to $\hat \p[\edot, t]$.
\end{enumerate}
\end{alg}
\end{framed}

As an alternative to the proposed two-stage procedure, the photoacoustic source could be recovered directly from  data (\ref{eq:data-cspat})   based on minimizing the  $\ell^1$-Tikhonov regularization  functional  \cite{GraHalSch08,Hal13b}
\begin{equation} \label{eq:sparse}
	\frac{1}{2} \norm{\y - (\A \circ  \Qo  ) \hat \source  }^2_2
	+
	\lambda \norm{  \bPsi   \hat \source  }_1  \to \min_{\hat \source}\,.
\end{equation}
Here $\bPsi $ is a suitable basis that sparsifies the photoacoustic source $\source$. However, such an approach  is numerically expensive since the three-dimensional wave equation  and its adjoint have  to be solved  repeatedly.  The proposed two-step reconstruction scheme is much faster because it avoids evaluating the wave equation, and the iterative reconstruction decouples into lower-dimensional problems for every $t$.  A simple estimation of the number of floating point operations (flops) reveals the dramatic speed improvement. Suppose we have $n =  N \times N$ detector locations, $\mathcal O (N)$ time instance and recover the source  on an $N \times N \times N$ spatial grid. Evaluation of a straight forward time domain discretization of $\Qo$  and its adjoint require $\mathcal O(N^5)$ flops. Hence, the iterative one-step reconstruction requires $N_{\rm iter} \, \mathcal O(N^5)$ operations, where $N_{\rm iter}$ is the number of iterations. On the other hand the  two-stage reconstruction requires  $ N_{\rm iter} \mathcal O (N^3 m)$ flops for the iterative data completion and additionally  $ \mathcal O(N^5)$ flops for the subsequent UBP reconstruction. In the implementation one takes the number of iterations (at least) in the order of $N$ and therefore the  two-step procedure is faster by at least one order of magnitude.

Compressed sensing schemes without  using random measurements  have been considered in~\cite{provost2009application,guo2010compressed,meng2012vivo}. In these approaches an  optimization problem of the form (\ref{eq:sparse}) is solved, where $\A$ is an under-sampled measurement matrix.  Especially when combined with a total variation penalty such approaches  yields  visually appealing result. Strictly taken, the measurements used there are not shown to yield compressed sensing,  which would require some form of incoherence between the measurement matrix and the sparsifying basis (usually established by  randomness).  For which class of phantoms  undersampled   point-wise  measurements have  compressed sensing capability for PAT is currently an unsolved problem.

\subsection{Sparsifying temporal transform}
\label{sec:sparse}

In order that the pressure data can be recovered
by (\ref{eq:ell1relax}) one requires  a
suitable  basis $\bPsi \in \R^{n \times n}$ such that
the pressure  is sparsely represented in this basis and that the composition $\A \circ \bPsi$ is a proper compressed sensing matrix. For expander matrices  these two conditions are not compatible.  To overcome this obstacle  in \cite{SanKraBerBurHal15,burgholzer2016}  we developed the concept of a sparsifying temporal transform for the two-dimensional case
in  circular geometry. Below we extend this concept to three spatial dimensions using combinations of point-wise pressure values.

Suppose we apply a transformation $\T$ to the data $t \mapsto \y[\edot , t]$ that only acts in the temporal variable. Because the  measurement matrix  $\A$ is applied in the spatial  variable,  the transformation $\T$ and the measurement matrix  commute, which yields
\begin{equation} \label{eq:data-spars}
	\T \y     =  \A  ( \T \p   ) \,.
\end{equation}
We call  $\T $  a sparsifying temporal transform, if   $\T \p \ekl{\edot, t} \in \R^n $ is sufficiently sparse for  a suitable class  of source distributions and all times $t$. In this work we propose the following sparsifying spatial transform
\begin{equation}\label{eq:sparstrafo}
           \T  ( \p )
            \coloneqq
            t^3 \partial_t  t^{-1} \partial_t t^{-1} \p\,.
\end{equation}
The sparsifying effect of this transform is illustrated in Figure~\ref{fig:sparse} applied to the pressure data  arising from  a uniform spherical  source.
The reason for  choice of  (\ref{eq:sparstrafo}) is as follows: It is well known that the pressure signals induced by a uniform absorbing sphere has an N-shaped  profile. Therefore, applying the second temporal derivative to $\p$ yields a signal that is sparse.  The modification of the second derivative is used because the term
$\partial_t t^{-1} \p$ appears in the universal backprojection and therefore only one numerical integration is required in the implementation of our approach. Finally, we empirically found that the leading factor $t^3$ results in well balanced peaks in Figure~\ref{fig:sparse} and yields good numerical results.

\begin{figure}[thb!]
\begin{centering}
\includegraphics[width=\columnwidth]{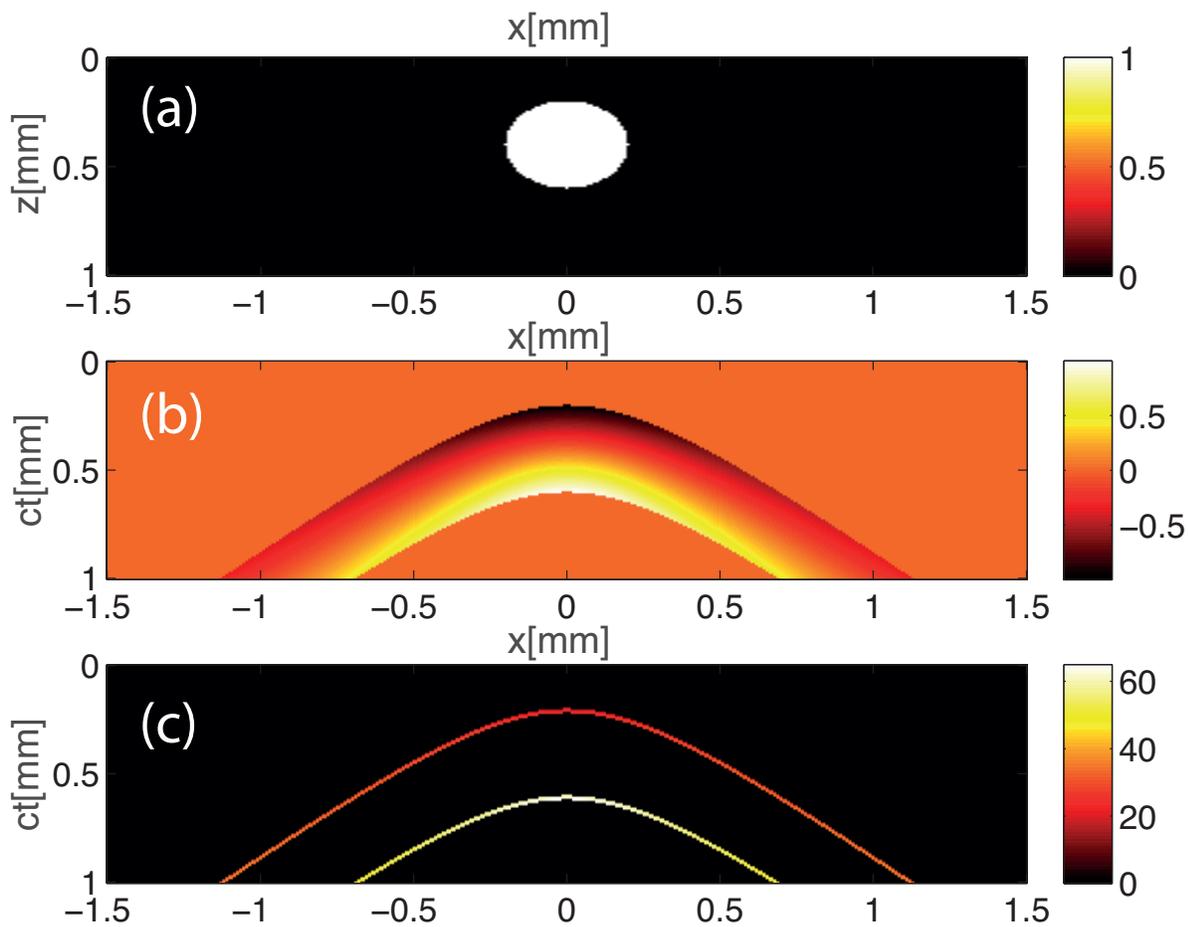}
\caption{\textsc{Effect of the sparsifying transform}. Top: Cross section of a uniform  spherical  source. Middle: Corresponding pressure data. Bottom: Result after applying the sparsifying transform $\T$. \label{fig:sparse} }
\end{centering}
\end{figure}

Having a sparsifying temporal transform at hand, we can construct the photoacoustic
source by the following modified two-stage approach.  In the first step recover
an approximation $\hat \q  [\edot,t] \simeq \T \p [\edot,t]$ by solving
\begin{equation} \label{eq:Tell1}
 \frac{1}{2} \norm{  \T \y[\edot,t] - \A \hat \q [\edot,t]      }^2
 +  \lambda   \norm{   \hat\q [\edot,t]  }_1   \to \min_{\hat \q}\,.
 \end{equation}
In the second step, we  recover the photoacoustic source by  implementing
the UBP expressed in terms  of the sparsified pressure,
 \begin{equation}\label{eq:ubp-mod}
           \source(\rr)    =
        - \frac{\ww}{\pi}
        \int_{\R^2}
        \int_{\abs{\rr-\rr_S}}^\infty
        ( t^{-3} \T    \Qo \source )(\uu_S,\vv_S, t)
        \rmd t
        \rmd S
 \,.
\end{equation}
Here $\rr = (\uu, \vv,\ww)$ is a reconstruction point and $\rr_S = (\uu_S,\vv_S, 0)$ a point on the measurement surface.
The modified UBP formula (\ref{eq:ubp-mod}) can be implemented analogously to Algorithm~\ref{alg:ubp}.  In summary, we obtain the following reconstruction algorithm.

\begin{framed}
\begin{alg}[Compressed sensing reconstruction with sparsifying temporal transform]\label{alg:cs2}\mbox{}\\
\underline{Goal:}
Reconstruct  $\source$ in~(\ref{eq:wave}) from data~(\ref{eq:data-cspat}).

\begin{enumerate}[label=(S\arabic*), leftmargin=2em]
\item\label{it:cs2a}
Recover sparsified point-measurements:
\begin{itemize}[leftmargin=1em]
\item 
Compute the filtered data $\T \y (t)$

\item 
Recover an approximation   $\hat \q[\edot, t] $\\
to  $\T \p[\edot, t]$  by  solving (\ref{eq:Tell1}).
\end{itemize}

\item \label{it:cs2b}
UBP algorithm for sparsified  data:
\begin{itemize}[leftmargin=1em]
\item
For any $i$, $\rho$ set\\
$\q[i, \rho ] \gets
\int_{\rho}^\infty t^{-3} \q [i, t]  \,  \rmd t $

\item
For  any $k$ set\\
$p_0^{\rm CS}[k] \gets \frac{v[k] }{\pi}
\sum_{i=1}^N \q\bigl[ i,  \abs{\rr[k]-\rr_S[i]} \bigr] w_{i}$.
\end{itemize}
\end{enumerate}
\end{alg}
\end{framed}
Since (\ref{eq:Tell1}) can be solved separately for  every $t$, the modified two-stage Algorithm \ref{alg:cs2} is again much faster than a  direct approach based on (\ref{eq:sparse}).
Moreover, from general recovery results in  compressed sensing presented in the Appendix,  \ref{alg:cs2} yields theoretical recovery guarantees for Bernoulli, subsampled Hadamard matrices as well as 
expander matrices (adjacency matrices of left $d$-regular graphs); see   Figure~\ref{fig:matrices}.

\begin{figure}[thb!]
\begin{centering}
\includegraphics[width=\columnwidth]{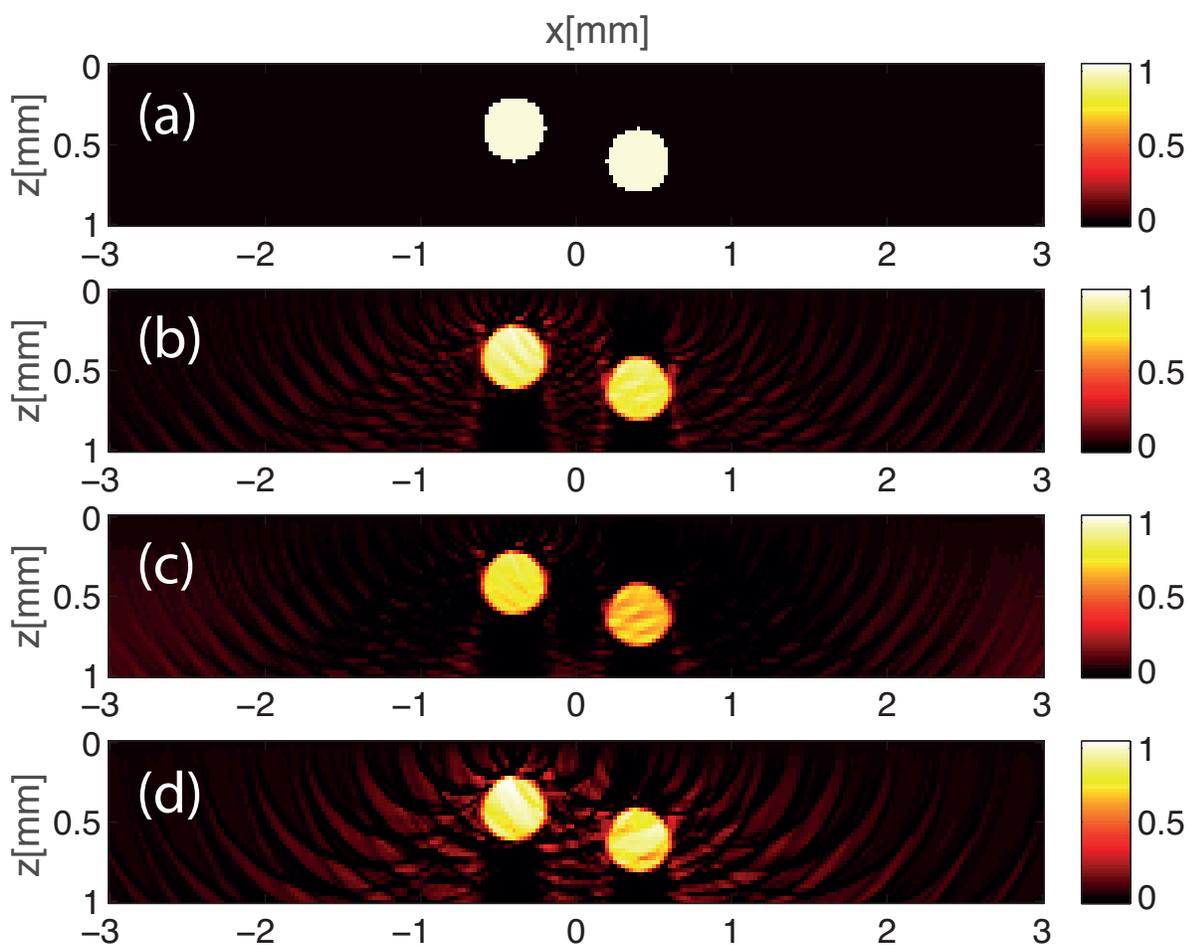}
\caption{\textsc{Three-dimensional compressed sensing  PAT versus  standard approach.}  (a)  Cross section of superposition of two uniform spherical absorbers. (b) Reconstruction using $4096$ point measurements on a Cartesian grid. (c) Compressed sensing reconstruction using $1024$ measurements with $d=15$. (d)  Reconstruction using $1024$ point measurements on a Cartesian grid.\label{fig:simu} }
\end{centering}
\end{figure}

\section{Numerical and experimental results}
\label{sec:num}

\subsection{Results for simulated data}

 We consider  reconstructing a  superposition  of two spherical absorbers,  having centers in the vertical plane  $\set{(\uu,\vv,\ww) \in \R^3 \mid \vv = 0}$. The vertical  cross section of the photoacoustic  source  is shown  in Figure~\ref{fig:simu}(a). In order to test our compressed sensing approach we first create point samples of  the pressure $\Qo \source$ on an equidistant Cartesian grid on the square $[-3, 3] \times [-3,3]$  using $64 \times 64$ grid points.
From that we compute compressed sensing data
\begin{equation} \label{eq:cs1}
	\y[j, t]
	=
	\sum_{i=1}^{4096}  \A\ekl{j,i} \p\ekl{i,t}
\quad \textnormal{ for } j \in \set{ 1, \dots,  1024 }\,.
\end{equation}
The choice $m = 1024$  corresponds to an reduction of measurements by a factor $4$. The expander matrix $\A$ was chosen as  the adjacency matrix of a randomly left $d$-regular graph with $d = 15$; see Example~\ref{ex:expander}
in the Appendix.
The pressure signals $\p\ekl{i,t}$ have been computed by the explicit
formula for the pressure of a uniformly absorbing sphere \cite{diebold1992photoacoustic} and evaluated at  $243$ times points $ct$ uniformly  distributed  in the interval $[0,6]$.

\begin{figure}[thb!]
\begin{centering}
\includegraphics[width=\columnwidth]{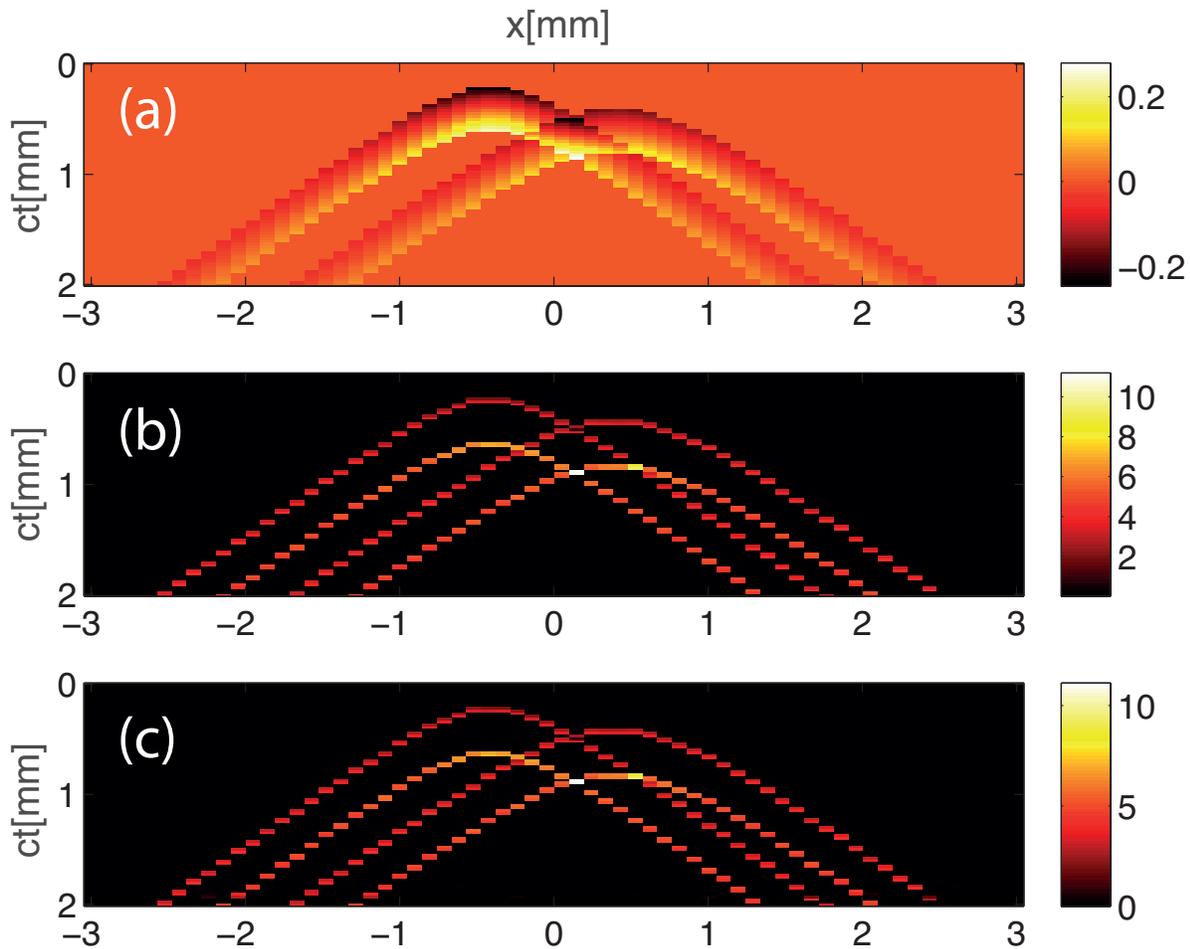}
\caption{\textsc{Result of sparse recovery.}   (a) Pressure at  $\ww = 0$ induced by two spherical absorbers shown in Figure
 \ref{fig:simu}. (b) Result  after applying the sparsifying     transform. (c) Reconstruction of the sparsified pressure from compressed sensing measurements   using $\ell^1$ minimization. \label{fig:sparserec} }
\end{centering}
\end{figure}

\begin{table}
\begin{tabular}{l | l | l | l}
\toprule
{}           & 4096 standard   & 1024 standard   & 1024 CS      \\ \midrule
$\alpha=1$   &   0.0472  &   0.0660 &  0.0409    \\  \midrule
$\alpha=2$   &   0.1046   &  0.1256 & 0.1124   \\  \bottomrule
 \end{tabular}
 \caption{Normalized $\ell^\alpha$-reconstruction errors   for $\alpha = 1, 2$.\label{tab:error} }
  \end{table}

Figure~\ref{fig:simu} shows the reconstruction results
using 4096 point samples  using Algorithm~\ref{alg:ubp} (Figure~\ref{fig:simu}(b)) and the reconstruction from 1024 compressed sensing measurements using Algorithm \ref{alg:cs2} (Figure~\ref{fig:simu}(c)). 
The reconstruction has been
computed at  $241 \times 41$ grid points in a vertical slice of size $[-3, 3] \times [0,1]$.
The $\ell^1$-minimization  problem (\ref{eq:Tell1}) has been solved
using the FISTA~\cite{BecTeb09}. For that purpose the matrix $\A$ has been rescaled to have 2-norm equal to one. The regularization parameter has then been set to $\lambda = 10^{-5}$ and we applied 7500 iterations of the
FISTA with maximal step size equal to one.
We see that the image quality from the compressed sensing reconstruction is comparable to the reconstruction from full data  using only a fourth of the number of measurements. For comparison purpose, Figure~\ref{fig:simu}(d) also shows the reconstruction using 1024 point samples. One clearly recognizes the increase of undersampling artifacts and worse image quality compared to the compressed sensing reconstruction using the same number of measurements. A more precise error evaluation is given in Table~\ref{tab:error}, where we show the normalized 
$\ell^\alpha$-error $\sqrt[\alpha]{ \sum_{k } |p_0[k]  - p_0^{\rm CS}[k]|^\alpha / N}$ for  $\alpha = 1$ and $\alpha=2$. The reconstruction error  in $\ell^1$-norm is
even slightly smaller for the compressed sensing reconstruction  than for the full reconstruction.  
This might be due to a slight denoising effect of $\ell^1$-minimization that removes some small amplitude errors
(contributing more to the more $\ell^1$-norm than to the $\ell^2$-norm).
 Figure~\ref{fig:sparserec} shows the pressure corresponding to the absorbers shown in Figure~\ref{fig:simu} together with the sparsified pressure and its reconstruction from compressed sensing data.

\begin{figure}[thb!]
\begin{centering}
\includegraphics[width=\columnwidth]{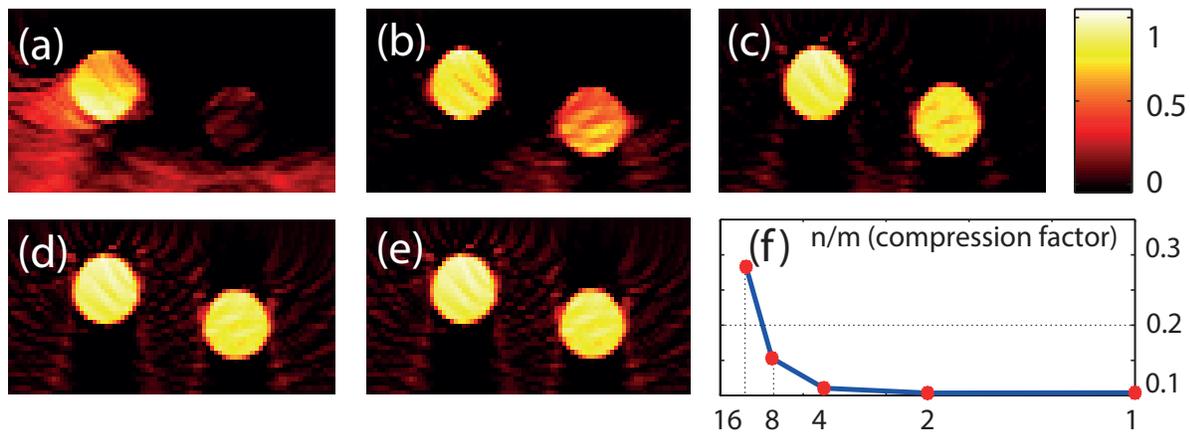}
\caption{\textsc{Recovery results for varying compression factor
$n/m$.}
(a) $n/m = 16$.
(b) $n/m = 8$.
(c) $n/m = 4$.
(d) $n/m = 2$.
(e) $n/m = 1$.
(f) Normalized $\ell^2$-reconstruction in dependence of the compression factor.
 \label{fig:series} }
\end{centering}
\end{figure}

Finally, Figure~\ref{fig:series} shows the reconstruction (restricted to $[-1, 1] \times [0,1]$) using Algorithm~\ref{alg:cs2} for varying compression
factors $n/m = 16, 8, 4, 2,1$. In all cases $d=15$ and $\lambda = 10^{-5}$ have been used and 7500 iterations of the  FISTA  have been applied.  As expected, the reconstruction error  increases with increasing compression factor. One further observes that the compression factor of 4 seems a good choice since  for smaller compression factor the error increases more severely. In our numerical studies (not shown) we observed that also for different  discretizations a compression factor of 4 is a good choice.

\begin{figure}[thb!]
\begin{centering}
\includegraphics[width=0.7\columnwidth]{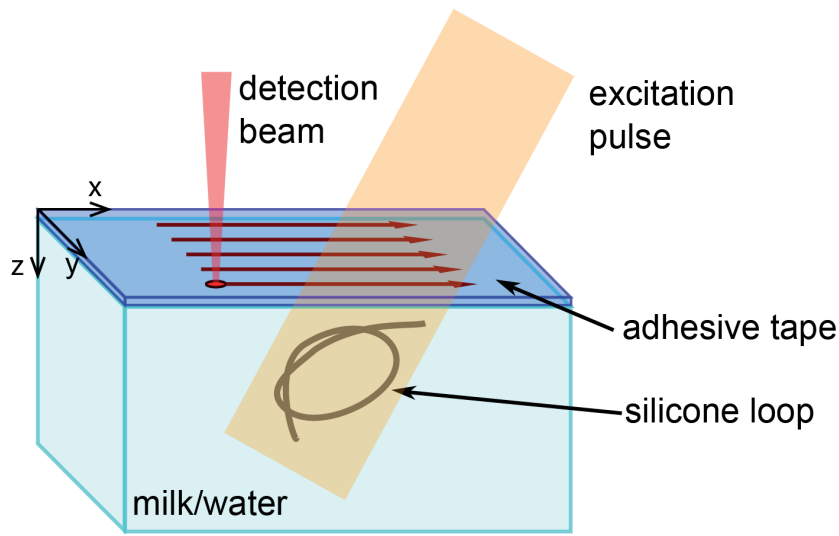}
\caption{\textsc{Schematic of experimental setup of non-contact photoacoustic imaging.}     Photoacoustic waves are excited by short laser pulses. The ultrasonic signals are measured on the surface of the sample using a non-contact photoacoustic imaging technique. \label{fig:setup} }
\end{centering}
\end{figure}

\subsection{Results for experimental data}

Experimental data have been  obtained from a silicone tube phantom as shown in Figure~\ref{fig:setup}.
The silicone tube was filled with black ink  (Pelikan 4001 brillant black, absorption coefficient of $\unit[54]{/cm}$ at $\unit[740]{nm}$), formed to a knot, and immersed in a milk/water emulsion. The outer and inner diameters of the tube were $\unit[600]{\mu m}$  and $\unit[300]{\mu m}$, respectively. Milk was diluted into the water to mimic the optical scattering properties of tissue; an adhesive tape, placed on the top of the water/milk emulsion, was used to mimic skin.   Photoacoustic signals were excited at a wavelength of  $\unit[740]{nm}$ with nanosecond pulses from an optical parametric oscillator pumped by a frequency doubled Nd:YAG laser.  The radiant expose was
$\unit[105]{Jm^{-2}}$, which is below the maximum permissible exposure for skin of $\unit[220]{Jm^{-2}}$. The resulting ultrasonic signals were detected on the adhesive tape by a non-contact photoacoustic imaging setup as described in \cite{Hochreiner13}. In brief, a continuous wave detection beam with a wavelength of $\unit[1550]{nm}$  was focused onto the sample surface.  The diameter of the focal spot was about $\unit[12]{\mu m}$. Displacements on the sample surface, generated by the impinging ultrasonic waves, change the phase of the reflected laser beam. By collecting and demodulating the reflected light, the phase information and, thus, information on the ultrasonic displacements at the position of the laser beam can be obtained. To allow three-dimensional measurements, the detection beam is raster scanned along the surface.  The obtained displacement data do not fulfill the wave equation and cannot be used for image reconstruction directly. Thus, to convert the displacement data to a quantity (roughly) proportional to the pressure, the first derivative in time of the data was calculated \cite{Berer10}.

\begin{figure}[thb!]
\begin{centering}\includegraphics[width=0.8\columnwidth]{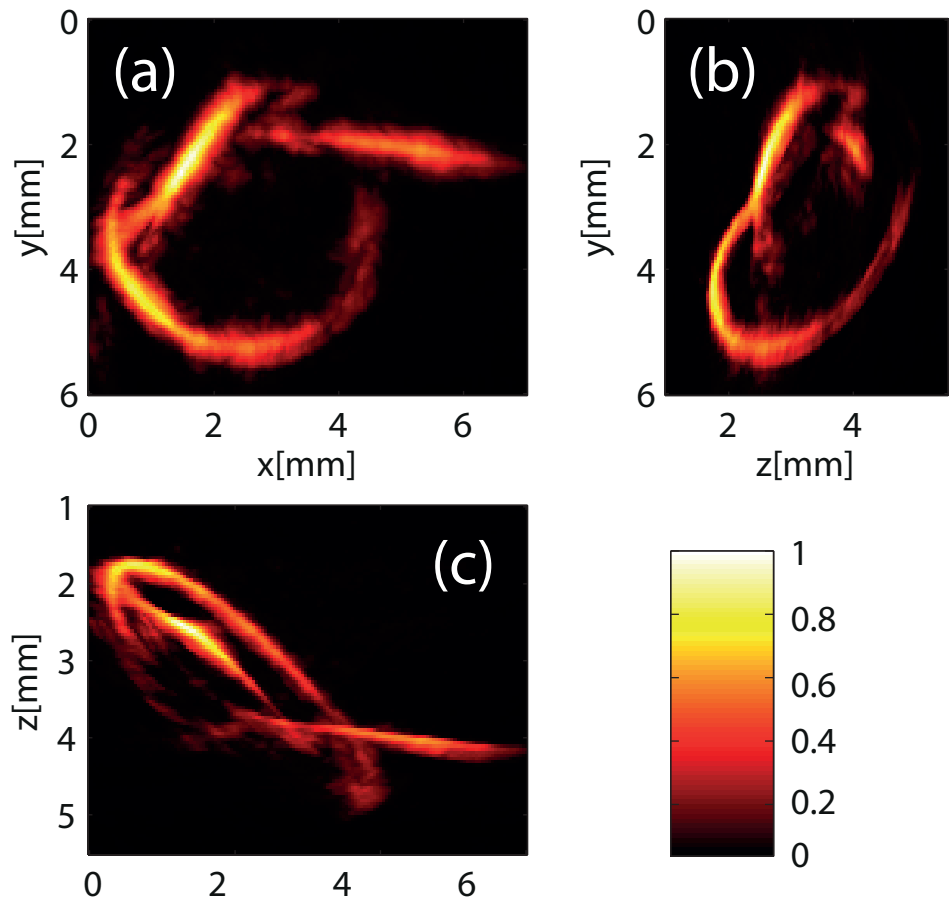}
\caption{\textsc{Reconstruction results using compressed sensing measurements.}  Maximum intensity projections of a silicone loop along the $z$-direction (a), the $x$-direction (b), and the $y$-direction (c). \label{fig:recrealCS}}
\end{centering}
\end{figure}

\begin{figure}[thb!]
\begin{centering}
\includegraphics[width=0.8\columnwidth]{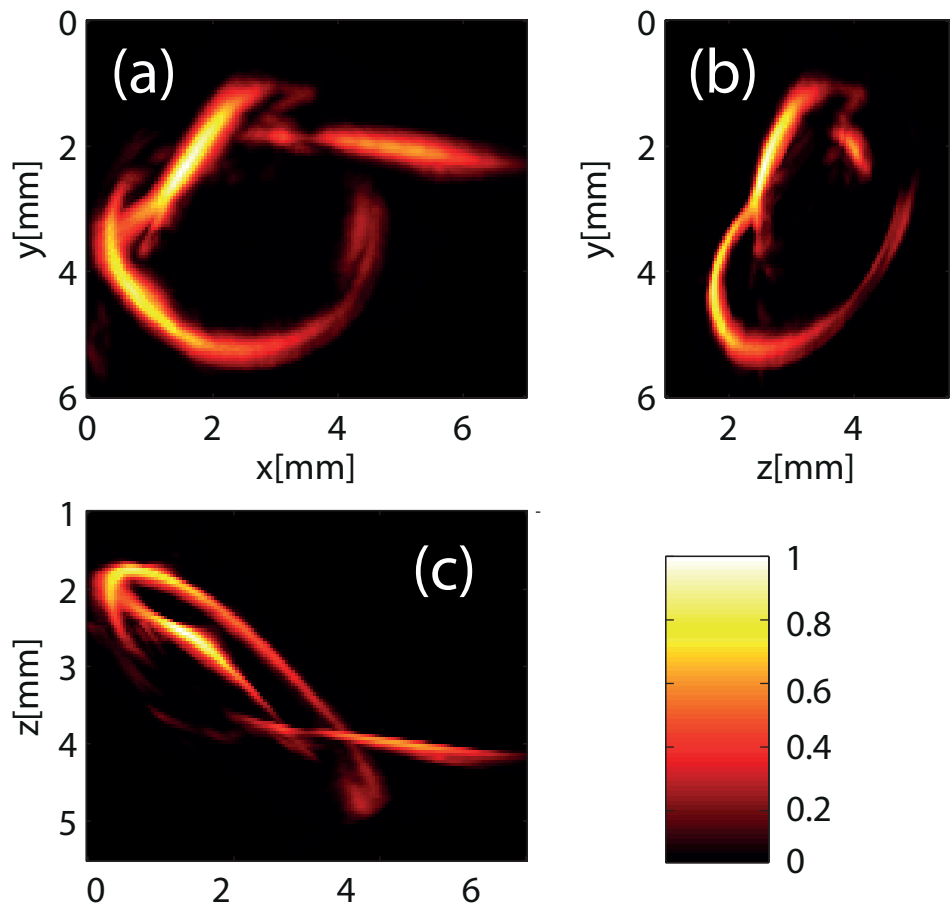}
\caption{\textsc{Reconstruction results using full  measurements.} Maximum intensity projections of a  silicone loop along the $z$-direction (a), the $x$-direction (b), and the $y$-direction (c).\label{fig:recreal}}
\end{centering}
\end{figure}

Using this setup, point-wise pressure data on the measurement  surface
have been collected for $4331 = 71 \times 61$ detector positions on   over an area of
$\unit[7]{mm} \times \unit[6]{mm}$. From this data we generated $m =  1116$ compressed sensing  measurements, where each detector location has been used  $d = 10$ times in total.
Figure~\ref{fig:recrealCS} shows the maximum amplitude projections  along the $z$, $x$, and $y$-direction, respectively, of the three-dimensional reconstruction  from compressed sensing data using Algorithm~\ref{alg:cs2}. The sparsified pressure  has been reconstructed  by minimizing (\ref{eq:Tell1}) with the FISTA  using 500 iterations and a regularization parameter of $10^{-5}$. Further, the  three-dimensional reconstruction  has been evaluated at  $110  \times 122 \times   142
$ equidistant grid points.
For  comparison purpose, in Figure~\ref{fig:recreal} we show the  maximum amplitude projections   from the UBP Algorithm~\ref{alg:ubp} applied to  the original data set. We observe  that there is only a small difference between the reconstructions in terms of  quality measures such as contrast, resolution and signal to noise ratio. Only, the structures in the compressed sensing reconstruction appear to be slightly less regular. A detailed quality evaluation  is beyond the scope of this paper, which aims at serving as proof of principle of our two-stage compressed sensing approach with sparsifying transform.
However the compressed sensing approach uses only a fourth of the number of measurements of the original data set.
This clearly demonstrates the potential of our compressed sensing scheme for decreasing the number of measurements while keeping the image quality.

\begin{figure}[thb!]
\begin{centering}\includegraphics[width=0.8\columnwidth]{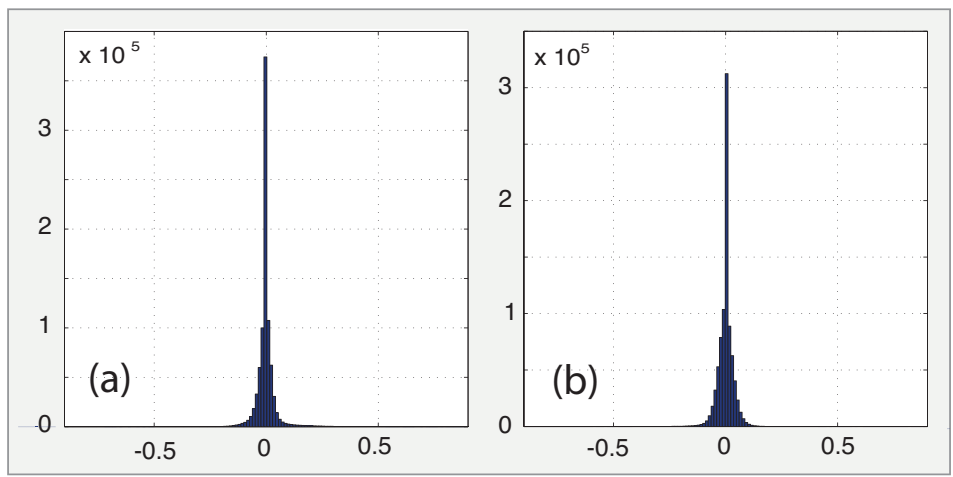}
\caption{\textsc{Histograms of experimental data.}
(a) Histogram for measured  pressure values (normalized to the interval $[0,1]$).
(b) The same after applying the sparsifying transform $\T$. \label{fig:hist}}
\end{centering}
\end{figure}

Figure~\ref{fig:hist} shows histograms of the pressure values before and
after applying the sparsifying temporal transform. In both cases the histograms are concentrated around the value zero. This implies  the  approximate sparsity and therefore justifies our approach,
even if the phantom is not a superposition of uniformly absorbing  spheres.
It further shows that for the  present situation one could even apply our two-stage
procedure without applying the  sparsifying transform.

\section{Discussion}
\label{sec:discussion}

To ensure sparsity in PAT, the  standard approach is choosing a suitable  basis
which sparsely represents  the pressure data on the measurement surface. In this paper we developed a different concept based
on sparsifying temporal transforms.  Since any temporal transform intertwines
with the spatial measurements our approach can be used in combination with any
measurement  matrix that is incoherent to the pixel basis.   This includes binary random matrices
such as Bernoulli, Hadamard, or expander matrices (see~\ref{sec:cs} for details).
According to the compressed sensing theory, expander matrices  can  be used  with
binary entries 0 and 1.
Bernoulli and Hadamard matrices, on the other hand,  should be used with zero mean
(achieved, for example taking $\pm 1$ as binary entries).
As  0/1 entries  can be practically  most simple be realized, for Bernoulli and Hadamard matrices
the mean value has to be subtracted after the measurements process~\cite{HuyZhaBetEtal14}.
Avoiding such additional data manipulations  is one reason why we currently work with  expander matrices.
Another reason is the sparse structure of expander matrices which can be used to accelerate image reconstruction.
In future work we will also investigate  the use of Bernoulli and Hadamard  in combination with sparsifying
temporal or spatial transforms, and compare the performance of these measurement ensembles
in different situations. 

As mentioned in the introduction, patterned interrogation can be used   to practically implement compressed sensing in PAT. It has been realized by using a digital micromirror device \cite{HuyZhaBetEtal14,huynh2016single}, where a Fabry-Perot sensor was illuminated by a wide-field collimated beam. The reflected beam, carrying the ultrasonic information on the acoustic field, was then sampled by the DMD and the spatially integrated response was measured by a photodiode.
Another possibility is the application of spatial light modulators (SLMs), which are able to modulate the phase of the light. By using such SLMs arbitrary interrogation patterns can be generated directly on a sample surface \cite{grunsteidl2013spatial}.  SLMs are commercially available for a wavelength of \unit[1550]{nm}, which is the most common wavelength used in optical detection schemes. However, also for other wavelengths appropriate devices are available. State-of-the-art SLMs provide typical resolutions between $1920 \times 1080$ pixels and $4094 \times 2464$ pixels, which is sufficient for the compressed imaging scheme presented in this work. For a resolution of $1920 \times 1080$, the typically achieved frame rate is \unit[60]{Hz}. This is faster than the pulse repetition rate of commonly used excitation laser sources for PAT, thus enabling single shot measurements. If a faster repetition rate is required, one could use SLMs with a higher frame rate. These, however, usually exhibit lower resolution.

For the Fabry-Perot etalon sensors, the wavelength of the interrogation beam has to be tuned, such that it corresponds to the maximum slope of the transfer function of the sensor. Since for the patterned interrogation scheme only one wavelength is used for the acquisition of the integrated response this demands high quality Fabry-Perot sensors with highly uniform sensor properties. For non-contact schemes, using Mach-Zehnder or Michelson based demodulation, the sensitivity does not depend on the wavelength as the sensor is stabilized by the relative phase between the interrogation beam and a reference beam.
 However, if the surface is not adequately flat, the phase of the reflected light is spatially varying. Thereby, the sensitivity changes over the detection surface, and even locations with zero or low response could exist. Here, the phase modulation capability of SLMs offers the possibility to compensate for this. In general, each pixel of a SLM can shift the phase of light at least up to $2\pi$ and the resulting phase distribution is impressed on the reflected beam. Separate lens functions can be applied to each detection point individually by using distinct kernels for each of these points \cite{curtis2002dynamic}. In case the shape of the sample surface is known, the phase at each detection point can be chosen to compensate for the phase shifts caused by the imperfect sample surface. With this method it is even possible to choose different focal distances for each detection point, so that detection on even rougher surfaces could be facilitated.

\section{Conclusion}
\label{sec:conclusion}

To speed up the data collection process in  sequential scanning PAT while  keeping sensitivity high without significantly increase the production costs, one has to reduce the number of spatial measurements. In this paper we proposed a compressed  sensing scheme for that purpose  using random measurements  in combination with a sparsifying temporal transform.  We presented a selected review of compressed sensing that demonstrates  the  role of sparsity and randomness for high resolution recovery. Using general results from compressed sensing we were able to derive  theoretical recovery guarantees for our approach based on  sparsifying temporal transforms. Further, this comes with a fast algorithmic implementation.

\section*{Acknowledgments}

This work has been supported by the Austrian Science Fund (FWF), project number P25584-N20, the Christian Doppler Research Association (Christian Doppler Laboratory for Photoacoustic Imaging and Laser Ultrasonics), the European Regional Development Fund (EFRE) in the framework of the EU-program Regio 13, the federal state Upper Austria.
S. Moon thanks University of Innsbruck for its hospitality during his visit. The work of S. Moon has been supported by the National Research Foundation of Korea grant funded by the Korea government (MSIP) (2015R1C1A1A01051674) and the TJ Park
Science Fellowship of POSCO TJ Park Foundation.

\appendix

\section{Ingredients from compressed sensing}
\label{sec:cs}

In  this section we present the basic ingredients of compressed sensing that   explains the choice of the measurement matrices and the role of sparsity in PAT. The  aim of compressed sensing is to stably recover a signal or image modeled by vector
$\x \in \R^n$ from measurements
\begin{equation} \label{eq:ip}
	\y  = \A \x + \ee \,.
\end{equation}
Here $\A \in \R^{m \times n}$ with $m \ll n$ is the   measurement matrix,  $\ee$ is an unknown error (noise) and $\y$  models the given noisy data.   The basic components that make compressed sensing possible are sparsity (or compressibility)  of the signal  $\x$ and some form of randomness in the measurement matrix $\A$.

\subsection{Sparsity and compressibility}

The first basic ingredient of compressed sensing is   sparsity, that is defined as follows.

\begin{definition}[Sparse signals] \label{def:sparsity}\mbox{}\\
Let $s \in \N$ and $\x \in \R^n$. The  vector $\x$ is called $s$-sparse,  if  $\norm{\x}_0 \coloneqq \sharp (\set{i  \in \set{1, \dots , n} \mid \x\ekl{i} \neq 0})  \leq s$. One informally calls  $\x$ sparse, if  it is  $s$-sparse for sufficiently small $s$.
\end{definition}

In Definition~\ref{def:sparsity}, $\sharp (S)$ stands for the number of elements in a set $S$. Therefore $\norm{\x}_0$ counts the
 number of non-zero entries in the vector $\x$.
In the mathematical sense  $\enorm_0$ is neither a  norm or a quasi-norm\footnote{A quasi-norm satisfies all axioms of a norm, except that the triangle inequality is replaced by the weaker inequality $\norm{\x_1+ \x_2} \leq K \kl{\norm{\x_1}+ \norm{\x_2}}$ for some constant $K \geq 1$.} but it is common to call $\enorm_0$ the $\ell^0$-norm.  It satisfies
$\norm{x}_0 = \lim_{p \downarrow 0} \norm{\x}_p^p$, where
\begin{equation} \label{eq:ellp}
	\norm{\x}_p
	\coloneqq
	\sqrt[p]{\sum_{i=1}^n \abs{ \x[i]}^p}
	\quad  \textnormal{ with $p >0$ }   \,,
\end{equation}
stands for the  $\ell^p$-norm.
Recall that $\enorm_p$ is indeed a norm for $p \geq 1$ and  a quasi-norm  for $p \in (0,1)$.

Signals of practical  interest are often not  sparse in the strict sense, but can be well approximated  by sparse vectors. For  that purpose we next define the $s$-term approximation error that can be used as a  measure for compressibility.

\begin{definition}[Best $s$-term approximation error]\label{def:compressibility}\mbox{}\\
Let $s \in \N$ and $\x \in \R^n$. One calls
\begin{equation*}
 	\sigma_{s} (\x)
	\coloneqq \inf \set{ \norm{\x - \x_s}_1
	\mid  \x_s \in \R^n \textnormal{ is $s$-sparse} }
\end{equation*}
the best  $s$-term approximation error of $\x$ (with respect to the $\ell^1$-norm).
\end{definition}

The best  $s$-term approximation error $\sigma_{s} (\x)$ measures,  in terms of the $\ell^1$-norm,  how much the vector $\x$ fails  to be $s$-sparse.
One  calls $\x \in \R^n$ compressible, if  $\sigma_s(\x)$ decays
sufficiently fast with increasing $s$.
The estimate (see \cite{FouRau13})
\begin{equation} \label{eq:approx}
	\sigma_{s} (\x)
	\leq \frac{q  (1-q)^{1/q-1} }{s^{1/q-1}}
	\norm{x}_q  \quad \textnormal{ for  }
	q  \in (0,1)
\end{equation}
shows that a signal is compressible if its $\ell^q$-norm is sufficiently small for some $q < 1$.

\subsection{The  RIP in compressed sensing}

Stable and robust recovery  of sparse vectors  requires the measurement matrix to well separate sparse vectors.
The RIP  guarantees   such a separation.

\begin{definition}[Restricted isometry property (RIP)]\mbox{}\\
Let $s \in \N$ and $\delta \in (0,1)$. The measurement matrix  $\A \in \R^{m\times n}$ is said to satisfy the RIP of order $s$ with constant $\delta$,  if, for all $s$-sparse $\x \in \R^n$,
\begin{equation}\label{eq:rip}
(1-\delta)\norm{\x}_2^2 \leq \norm{\A\x}_2^2 \leq (1+\delta) \norm{\x}_2^2 \,.
\end{equation}
We write $\delta_s$ for the smallest constant  satisfying (\ref{eq:rip}).
\end{definition}

In the  recent years, many sparse recovery results have been derived  under various forms of the RIP.
Below we give a result  derived recently in \cite{CaiZha13}.

\begin{theorem}[Sparse recovery under the RIP] \label{thm:rip} \mbox{}\\
Let $\x \in \R^n$ and let $\y \in \R^m$ satisfy  $\norm{ \y - \A\x }_2 \leq \eps$ for some  noise  level $\eps >0$. Suppose that $\A \in \R^{m\times n}$ satisfies the RIP of order $2s$ with constant $\delta_{2s} < 1/2$, and let $\x_\star$ solve
\begin{eqnarray} \nonumber
&&\textnormal{minimize}_{\z} \norm{\z}_1 \\ \label{eq:ell12}
&&\textnormal{such that } \norm{\A \z - \y}_2
\leq \epsilon \,.
\end{eqnarray}
Then, for constants $c_1, c_2$ only depending on $\delta_{2s}$,
$\| \x - \x_\star \|_2
\leq   c_1  \sigma_{s}(\x) /  \sqrt{s}
+  c_2 \epsilon$.
\end{theorem}

\begin{proof}
See~\cite{CaiZha13}.
\end{proof}

Theorem \ref{thm:rip} states stably and robust  recovery  for measurement matrices satisfying the RIP.
The error estimate consists of two terms: $c_2 \epsilon$ is due to the data noise and is proportional to the noise  level (stability with respect to noise). The term $c_1  \sigma_{s}(\x) /  \sqrt{s}$ accounts for the  fact that the unknown may not be strictly
$s$-sparse  and shows robustness with respect to the model assumption of sparsity.

No deterministic construction is known providing large measurement matrices satisfying the RIP. However,  several types of random matrices  are known to satisfy the  RIP with high probability. Therefore,   for such measurement matrices, Theorem~\ref{thm:rip} yields stable and robust recovery  using (\ref{eq:ell12}). We give two important examples of  binary random matrices satisfying the RIP~\cite{FouRau13}.

\begin{example}[Bernoulli matrices]\label{ex:bernoulli}\mbox{}\\
A binary random  matrix $\B_{m,n} \in \set{-1,1}^{m\times n}$  is called  Bernoulli matrix  if its entries are independent and take the  values $-1$ and $1$ with equal probability. A Bernoulli matrix satisfies $\delta_{2s} < \delta $ with probability tending to $1$ as $m \to \infty$, if
\begin{equation}\label{eq:m-bernoulli}
	m \geq  C_\delta  s \kl{ \log ( n/s) + 1 }
\end{equation}
for  some  constant $C_\delta > 0$.
Consequently,  Bernoulli-measurements
yield stable and robust recovery by (\ref{eq:ell12}) provided that (\ref{eq:m-bernoulli}) is satisfied.
 \end{example}

Bernoulli matrices are dense and unstructured. If $n$  is large then storing and applying  such a matrix is expensive. The next example gives a structured binary matrix satisfying the RIP.

\begin{example} [Subsampled  Hadamard matrices]\label{ex:hadamard}\mbox{}\\
Let $n$ be   a power of two. The  Hadamard matrix $\Ho_n$ is a binary orthogonal  and self-adjoint $n \times n$ matrix that takes values in $\set{-1,1}$. It  can be defined inductively by $\Ho_1 = 1$ and
\begin{equation} \label{eq:hadamard}
 \Ho_{2n} \coloneqq
\frac{1}{\sqrt{2}}
\left[
\begin{array}{ll}
\Ho_{n} &  \phantom{-} \Ho_{n}\\
\Ho_{n}  & -\Ho_{n}
\end{array}
\right]
\,.\end{equation}
Equation (\ref{eq:hadamard}) also serves as the basis for evaluating $\Ho_n \x $ with $n \log n$ floating point operations.
A randomly subsampled  Hadamard matrix has  the form $\Ps_{m,n}  \Ho_n\in \set{-1,1}^{m\times n }$, where $\Ps_{m,n}$  is a subsampling operator  that selects $m$  rows uniformly at random. It  satisfies $\delta_{2s} <  \delta  $ with probability tending to $1$ as $n \to \infty$, if
\begin{equation}\label{eq:m-hadamard}
	m \geq  D_\delta  s   \log ( n)^4
\end{equation}
for  some constant $D_\delta > 0$.
Consequently,  randomly subsampled  Hadamard matrices again yield stable and robust recovery using (\ref{eq:ell12}).
\end{example}

\subsection{Compressed sensing using lossless expanders}

A particularly useful  type of binary measurement  matrices  for compressed sensing are sparse matrices
having exactly $d$ ones in each column.  Such a measurement matrix can be interpreted as  the adjacency  matrix of a left $d$-regular bipartite graph.

Consider the bipartite graph $(L, R, E)$ where  $L \coloneqq \set{1, \dots, n}$ is the set of left vertices, $R \coloneqq  \set{1, \dots, m}$ the set of right vertices and $E \subseteq  L \times R$  the set of edges. Any element $(i, j) \in E$ can be interpreted as a edge joining  vertices $i$ and $j$.
We write
$$N(I)  \coloneqq \set{ j \in R \mid  \exists i \in I \textnormal{ with }  \kl{i, j} \in E }$$
for the set of (right) neighbors of  $I \subseteq L$.

\begin{definition}[Left\label{def:left} $d$-regular graph]\mbox{}\\
The bipartite graph $(L, R, E)$ is called $d$-left regular, if
$\sharp \ekl{ N(\set{i})} =d$ for every $i \in L$.
\end{definition}

According to Definition \ref{def:left},    $(L, R, E)$ is  left $d$-regular if
any left vertex is connected to exactly $d$ right vertices.
 Recall that the adjacency matrix $\A \in \set{0,1}^{m \times n}$ of  $(L, R, E)$ is defined by $\A\ekl{j,i} =1$ if $(i, j) \in E$ and $\A\ekl{j,i} = 0$ if $(i, j) \not\in E$. Consequently the adjacency  matrix of a  $d$-regular graph contains exactly $d$ ones in each column. If $d$ is small, then  the adjacency  matrix of a left $d$-regular bipartite graph is sparse.

\begin{definition}[Lossless expander]\mbox{}\\
Let $s \in \N$ and $\theta \in (0,1)$.
A $d$-left regular graph $(L, R, E)$  is called an
$(s,d,\theta)$-lossless expander, if
\begin{equation}\label{eq:expander}
 \sharp \ekl{ N(I) }\geq  (1- \theta) \, d \, \sharp \ekl{I}
 \textnormal{ for } I \subseteq L
 \textnormal{ with  } \sharp \ekl{I} \leq s  \,. \end{equation}
We write $\theta_s$ for the smallest constant  satisfying (\ref{eq:expander}).
\end{definition}

It is clear that the adjacency  matrix of a  $d$-regular graph  satisfies $\sharp \ekl{ N(I) }\leq  d \, \sharp \ekl{I}$. Hence an  expander graph satisfies the two sided estimate $ (1- \theta) \, d \, \sharp \ekl{I} \leq \sharp \ekl{ N(I) }\leq d \, \sharp \ekl{I}$.   Opposed to   Bernoulli and subsampled Hadamard matrices,  a lossless expander does  not satisfy the $\ell^2$-based RIP. However, in such a situation,  one can use the following  alternative recovery  result.

\begin{theorem}[Sparse recovery for lossless expander] \label{thm:expander} \mbox{}\\
Let $\x \in \R^n$ and let $\y \in \R^m$ satisfy  $\norm{ \y - \A\x }_1 \leq \eps$ for some noise level $\eps >0$. Suppose that  $\A$ is the adjacency matrix  of a $(2s,d,\theta_{2s})$-lossless expander having $\theta_{2s} < 1/6$   and let $\x_\star$ solve
\begin{eqnarray}\nonumber
&&\textnormal{minimize}_{\z} \norm{\z}_1 \\  \label{eq:ell11}
&& \textnormal{such that } \norm{\A \z - \y}_1
\leq \epsilon \,.
\end{eqnarray}
Then, for constants $c_1, c_2$ only depending on $\theta_{2s}$, we have
$\| \x - \x_\star \|_1
\leq   c_1  \sigma_{s}(\x)
+   c_2 \eps / d$.
\end{theorem}

\begin{proof}
See~\cite{BerGilIndKarTra08,FouRau13}.
\end{proof}

Choosing a $d$-regular bipartite  graph uniformly at random yields a lossless  expander with high probability.
Therefore, Theorem \ref{thm:expander} yields stable and robust recovery for such type of random matrices.

\begin{example}[Expander matrix]\label{ex:expander}\mbox{}\\
Take $\A \in \{0,1\}^{m\times n}$ as the adjacency matrix  of a randomly chosen left $d$-regular bipartite graph. Then $\A$ has exactly $d$ ones in each column, whose locations are uniformly  distributed. Suppose further that for some constant $c_\theta$  only depending on $\theta$ the parameters $d$ and $m$ have been selected according to
\begin{eqnarray*}
m &\geq& c_\theta s \kl{ \log  (n / s) + 1} \\
d  &=&  \left\lceil \frac{2 \log (n / s ) + 2 }{ \theta} \right \rceil
\,.
\end{eqnarray*}
Then,  $\theta_s \leq \theta$ with probability tending to 1 as
$n \to \infty$. Consequently,  for the adjacency matrix of a randomly chosen left $d$-regular bipartite graphs, called
expander matrix, we have  stable and robust recovery by (\ref{eq:ell11}).
\end{example}


\section*{References}


\begin{thebibliography}{10}

\bibitem{Bea11}
P.~Beard.
\newblock Biomedical photoacoustic imaging.
\newblock {\em Interface focus}, 1(4):602--631, 2011.

\bibitem{Wan09b}
L.~V. Wang.
\newblock Multiscale photoacoustic microscopy and computed tomography.
\newblock {\em Nature Phot.}, 3(9):503--509, 2009.

\bibitem{XuWan06}
M.~Xu and L.~V. Wang.
\newblock Photoacoustic imaging in biomedicine.
\newblock {\em Rev. Sci. Instruments}, 77(4):041101 (22pp), 2006.

\bibitem{ZhaLauBea08}
E.~Zhang, J.~Laufer, and P.~Beard.
\newblock Backward-mode multiwavelength photoacoustic scanner using a planar
  fabry-perot polymer film ultrasound sensor for high-resolution
  three-dimensional imaging of biological tissues.
\newblock {\em Appl. Opt.}, 47(4):561--577, 2008.

\bibitem{Berer10}
T.~Berer, A~Hochreiner, S.~Zamiri, and P~Burgholzer.
\newblock Remote photoacoustic imaging on solid material using a two-wave
  mixing interferometer.
\newblock {\em Opt. Lett.}, 35(24):4151--4153, 2010.

\bibitem{Berer15}
T.~Berer, A.~Leiss-Holzinger, E.~Hochreiner, J.~Bauer-Marschallinger, and
  A.~Buchsbaum.
\newblock Multimodal non-contact photoacoustic and optical coherence tomography
  imaging using wavelength-division multiplexing.
\newblock {\em J. Biomed. Opt.}, 20(4):046013, 2015.

\bibitem{Eom15}
J.~Eom, S.~Park, and B.~B. Lee.
\newblock Noncontact photoacoustic tomography of in vivo chicken
  chorioallantoic membrane based on all-fiber heterodyne interferometry.
\newblock {\em J. Biomed. Opt.}, 20(4):106007, 2015.

\bibitem{huynh2016photoacoustic}
N.~Huynh, O.~Ogunlade, E.~Zhang, B.~Cox, and P.~Beard.
\newblock Photoacoustic imaging using an 8-beam fabry-perot scanner.
\newblock {\em Proc. SPIE}, 9708:97082L, 2016.

\bibitem{HuyZhaBetEtal14}
N.~Huynh, E.~Zhang, M.~Betcke, S.~Arridge, P.~Beard, and B.~Cox.
\newblock Patterned interrogation scheme for compressed sensing photoacoustic
  imaging using a fabry perot planar sensor.
\newblock {\em Proc. SPIE}, 8943:894327--5, 2014.

\bibitem{huynh2016single}
N.~Huynh, E.~Zhang, M.~Betcke, S.~Arridge, P.~Beard, and B.~Cox.
\newblock Single-pixel optical camera for video rate ultrasonic imaging.
\newblock {\em Optica}, 3(1):26--29, 2016.

\bibitem{SanKraBerBurHal15}
M.~Sandbichler, F.~Krahmer, T.~Berer, P.~Burgholzer, and M.~Haltmeier.
\newblock A novel compressed sensing scheme for photoacoustic tomography.
\newblock {\em SIAM J. Appl. Math.}, 75(6):2475--2494, 2015.

\bibitem{burgholzer2016}
P.~Burgholzer, M.~Sandbichler, F.~Krahmer, T.~Berer, and M.~Haltmeier.
\newblock Sparsifying transformations of photoacoustic signals enabling
  compressed sensing algorithms.
\newblock {\em Proc. SPIE}, 9708:970828--8, 2016.

\bibitem{CanRomTao06a}
E.~J. Cand{\`e}s, J.~Romberg, and T.~Tao.
\newblock Robust uncertainty principles: exact signal reconstruction from
  highly incomplete frequency information.
\newblock {\em IEEE Trans. Inf. Theory}, 52(2):489--509, 2006.

\bibitem{CanTao06}
E.~J. Cand{\`e}s and T.~Tao.
\newblock Near-optimal signal recovery from random projections: universal
  encoding strategies?
\newblock {\em IEEE Trans. Inf. Theory}, 52(12), 2006.

\bibitem{Don06}
D.~L. Donoho.
\newblock Compressed sensing.
\newblock {\em IEEE Trans. Inf. Theory}, 52(4):1289--1306, 2006.

\bibitem{CanDemDonYin06}
E.~Cand{\`e}s, L.~Demanet, D.~Donoho, and L.~Ying.
\newblock Fast discrete curvelet transforms.
\newblock {\em Multiscale Model. Sim.}, 5:861--899, 2006.

\bibitem{Mal09}
S.~Mallat.
\newblock {\em A wavelet tour of signal processing: The sparse way}.
\newblock Elsevier/Academic Press, Amsterdam, third edition, 2009.

\bibitem{Hochreiner13}
A.~Hochreiner, J.~Bauer-Marschallinger, B.~Burgholzer, P.~Jakoby, and T.~Berer.
\newblock Non-contact photoacoustic imaging using a fiber based interferometer
  with optical amplification.
\newblock {\em Biomed. Opt. Express}, 4(11):2322--2331, 2013.

\bibitem{XuWan05}
M.~Xu and L.~V. Wang.
\newblock Universal back-projection algorithm for photoacoustic computed
  tomography.
\newblock {\em Phys. Rev. E}, 71(1):016706, 2005.

\bibitem{BurBauGruHalPal07}
P.~Burgholzer, J.~Bauer-Marschallinger, H.~Gr{\"u}n, M.~Haltmeier, and
  G.~Paltauf.
\newblock Temporal back-projection algorithms for photoacoustic tomography with
  integrating line detectors.
\newblock {\em Inverse Probl.}, 23(6):S65--S80, 2007.

\bibitem{Hal13a}
M.~Haltmeier.
\newblock Inversion of circular means and the wave equation on convex planar
  domains.
\newblock {\em Comput. Math. Appl.}, 65(7):1025--1036, 2013.

\bibitem{Nat12}
F.~Natterer.
\newblock Photo-acoustic inversion in convex domains.
\newblock {\em Inverse Probl. Imaging}, 6(2):315--320, 2012.

\bibitem{Kun07a}
L.~A. Kunyansky.
\newblock Explicit inversion formulae for the spherical mean {R}adon transform.
\newblock {\em Inverse Probl.}, 23(1):373--383, 2007.

\bibitem{Hal14}
M.~Haltmeier.
\newblock Universal inversion formulas for recovering a function from spherical
  means.
\newblock {\em SIAM J. Math. Anal.}, 46(1):214--232, 2014.

\bibitem{HalPer15b}
M.~Haltmeier and S.~Pereverzyev, Jr.
\newblock The universal back-projection formula for spherical means and the
  wave equation on certain quadric hypersurfaces.
\newblock {\em J. Math. Anal. Appl.}, 429(1):366--382, 2015.

\bibitem{GraHalSch08}
M.~Grasmair, M.~Haltmeier, and O.~Scherzer.
\newblock Sparse regularization with {$l^q$} penalty term.
\newblock {\em Inverse Probl.}, 24(5):055020, 13, 2008.

\bibitem{Hal13b}
M.~Haltmeier.
\newblock Stable signal reconstruction via $\ell^1$-minimization in redundant,
  non-tight frames.
\newblock {\em IEEE Trans. Signal Process.}, 61(2):420--426, 2013.

\bibitem{provost2009application}
J.~Provost and F.~Lesage.
\newblock The application of compressed sensing for photo-acoustic tomography.
\newblock {\em IEEE Trans. Med. Imag.}, 28(4):585--594, 2009.

\bibitem{guo2010compressed}
Z.~Guo, C.~Li, L.~Song, and L.~V. Wang.
\newblock Compressed sensing in photoacoustic tomography in vivo.
\newblock {\em J. Biomed. Opt.}, 15(2):021311--021311, 2010.

\bibitem{meng2012vivo}
J.~Meng, L.~V. Wang, D.~Liang, and L.~Song.
\newblock In vivo optical-resolution photoacoustic computed tomography with
  compressed sensing.
\newblock {\em Optics letters}, 37(22):4573--4575, 2012.

\bibitem{diebold1992photoacoustic}
G.~J. Diebold, T.~Sun, and M.~I. Khan.
\newblock Photoacoustic monopole radiation in one, two, and three dimensions.
\newblock {\em Phys. Rev. Lett.}, 67:3384--3387, Dec 1991.

\bibitem{BecTeb09}
A.~Beck and M.~Teboulle.
\newblock A fast iterative shrinkage-thresholding algorithm for linear inverse
  problems.
\newblock {\em SIAM J. Imaging Sci.}, 2(1):183--202, 2009.

\bibitem{grunsteidl2013spatial}
C.~Gr\"unsteidl, I.~A. Veres, J.~Roither, P.~Burgholzer, T.~W. Murray, and
  T.~Berer.
\newblock Spatial and temporal frequency domain laser-ultrasound applied in the
  direct measurement of dispersion relations of surface acoustic waves.
\newblock {\em Applied Physics Letters}, 102(1):011103, 2013.

\bibitem{curtis2002dynamic}
Jennifer~E. Curtis, Brian~A. Koss, and David~G. Grier.
\newblock Dynamic holographic optical tweezers.
\newblock {\em Optics Communications}, 207(1--6):169--175, 2002.

\bibitem{FouRau13}
S.~Foucart and H.~Rauhut.
\newblock {\em A mathematical introduction to compressive sensing}.
\newblock Springer, 2013.

\bibitem{CaiZha13}
T.~T. Cai and A.~Zhang.
\newblock Sharp {RIP} bound for sparse signal and low-rank matrix recovery.
\newblock {\em Appl. Comput. Harmon. Anal.}, 35(1):74--93, 2013.

\bibitem{BerGilIndKarTra08}
R.~Berinde, A.~C. Gilbert, P.~Indyk, H.~Karloff, and M.~J. Strauss.
\newblock Combining geometry and combinatorics: A unified approach to sparse
  signal recovery.
\newblock In {\em 46th Annual Allerton Conference on Communication, Control,
  and Computing, 2008}, pages 798--805, 2008.

\end{thebibliography}
\end{document}